\pdfoutput=1
\documentclass[a4paper,11pt]{article}
\usepackage[english]{babel}
\usepackage{amsfonts}
\usepackage{amsthm}
\usepackage{amsmath}
\usepackage{amssymb}
\usepackage{enumerate}
\usepackage{mathtools}
\usepackage{tikz}
\usepackage{tikz-cd}
\usepackage[all]{xy}
\usepackage{graphicx}
\usepackage{amsmath}
\usepackage{adjustbox}
\usepackage{graphicx}
\usepackage{extpfeil}
\usepackage{comment}
\usepackage{float}
\usepackage[labelformat=empty]{caption}
\usepackage[toc]{appendix}
\usepackage[left=3.2cm,top=3cm,right=3.2cm,bottom=3cm]{geometry}
\usepackage{hyperref}
\usepackage{resizegather}
\usepackage[activate={true, nocompatibility}, final, tracking=true, kerning=true, spacing=true, factor=1100, stretch=10, shrink=10]{microtype}
\usepackage{authblk}
\usepackage{cleveref}

\theoremstyle{plain}
\newtheorem{thm}{Theorem}[section]

\newtheorem{defn}[thm]{Definition}
\newtheorem{lemma}[thm]{Lemma}

\newtheorem{prop}[thm]{Proposition}

\newtheorem*{thm*}{Main Theorem}
\newtheorem*{coroll*}{Corollary}
\newtheorem*{prop*}{Proposition}

\theoremstyle{definition}
\newtheorem{example}[thm]{Example}
\newtheorem{remark}[thm]{Remark}

\newcommand{\Bun}{\text{Bun}}

\tikzset{
  symbol/.style={
    draw=none,
    every to/.append style={
      edge node={node [sloped, allow upside down, auto=false]{$#1$}}}
  }
}

\microtypecontext{spacing=nonfrench}

\begin{document}
\title{\textbf{On automorphisms of semistable $G$-bundles with decorations}}
\author{Andres Fernandez Herrero}
\date{}
\maketitle
\begin{abstract}
    We prove a rigidity result for automorphisms of points of certain stacks admitting adequate moduli spaces. It encompasses as special cases variations of the moduli of $G$-bundles on a smooth projective curve for a reductive algebraic group $G$. For example, our result applies to the stack of semistable $G$-bundles, stacks of semistable Hitchin pairs, and stacks of semistable parabolic $G$-bundles. Similar arguments apply to Gieseker semistable $G$-bundles in higher dimensions. We present two applications of the main result. First, we show that in characteristic $0$ every stack of semistable decorated $G$-bundles admitting a quasiprojective good moduli space can be written naturally as a $G$-linearized global quotient $Y/G$, so the moduli problem can be interpreted as a GIT problem. Secondly, we give a proof that the stack of semistable meromorphic $G$-Higgs bundles on a family of curves is smooth over any base in characteristic $0$.
\end{abstract}
\begin{section}{Introduction}
Let $G$ be a reductive algebraic group over a field $k$. Several moduli problems of interest parametrize principal $G$-bundles on a projective variety $X$ along with some additional data. Examples of such data include Hitchin pair structures \cite[2.7.4]{schmitt-decorated} (generalizing principal Higgs bundles), sections of an associated vector bundle, and sections of projective fibrations (e.g. $G$-bundles with parabolic structure at finitely many points \cite{heinloth-schmitt}). These are all particular examples of what we refer to as ``$G$-bundles with decorations". 

In order to construct a moduli space of such objects, it is often necessary to impose certain semistability conditions which help to rigidify the moduli problem. In this article we study the automorphisms of a semistable $G$-bundle with decoration. Recall that there is an algebraic group consisting of those automorphisms of the $G$-bundle that preserve the decoration. It is reasonable to expect that semistability imposes some rigidity on this group. One way to approach this is to restrict any such automorphism to the fiber of the $G$-bundle over a fixed closed point in $X$. The expectation is that not much global information of the automorphism is lost via this restriction.

In this paper we prove that the kernel of restriction from the algebraic group of decoration-preserving automorphisms to any fiber is a finite group scheme with unipotent geometric points. This shows that if $\text{char}(k) = 0$, then the restriction to any fixed closed point on the variety is an immersion of algebraic groups, and therefore no information is lost by restriction.

Let us be more precise about our setup. Let $X$ be a geometrically reduced and geometrically connected projective scheme over $k$. We work with a quasiseparated algebraic stack $\mathcal{M}$ locally of finite type over $k$, and equipped with a representable morphism $\varphi: \mathcal{M} \to \text{Bun}_{G}(X)$ into the moduli stack $\text{Bun}_{G}(X)$ of $G$-bundles on $X$. Some of the properties of the semistable locus of a moduli problem can be encoded by requiring that $\mathcal{M}$ admits an adequate moduli space in the sense of \cite{alper_adequate}. Under this assumption, we prove in Theorem \ref{thm: main theorem} a precise version of the rigidity result outlined in the previous paragraph.

For any $k$-point $p \in \mathcal{M}(k)$, the algebraic group $\text{Aut}(p)$ of automorphisms of $p$ embeds into the group $\text{Aut}(\varphi(p))$ of automorphisms of the underlying $G$-bundle $\varphi(p)$. Given any fixed point $x \in X(k)$, we can restrict elements of $\text{Aut}(p)$ to the group of automorphisms $\text{Aut}(\varphi(p))|_{x})$ of the $x$-fiber of the underlying $G$-bundle. This defines a restriction homomorphism $\text{res}_x: \text{Aut}(p) \to \text{Aut}(\varphi(p)|_x)$ of algebraic groups.
\begin{thm*}[= Theorem \ref{thm: main theorem}]
With assumptions and notation as above, for any $p \in \mathcal{M}(k)$ and any $x \in X(k)$, the kernel of the restriction morphism $\mathrm{res}_x: \mathrm{Aut}(p) \to \mathrm{Aut}(\varphi(p)|_x)$ is a finite group scheme with unipotent geometric points. In particular $\mathrm{res}_x$ is a closed immersion if $\mathrm{char}(k) = 0$.
\end{thm*}
This in particular gives a necessary condition for a decorated $G$-bundle to be a point of a stack that admits an adequate moduli space. Our proof of the main theorem proceeds by degenerating to a closed point of the stack and then using semicontinuity results to obtain information about the kernels of restriction morphisms.

In the special case when $X$ is a smooth curve, we discuss several examples where Theorem \ref{thm: main theorem} applies directly. This includes the moduli of semistable $G$-bundles \cite{ramanathan-moduli-git, ramanathan-git1, glss.singular.char}, semistable Hitchin pairs \cite[2.7.4]{schmitt-decorated}, semistable parabolic $G$-bundles \cite{heinloth-schmitt}, and semistable holomorphic chains \cite{geometry-hol-chains}. More generally, the main theorem can be applied to any other moduli problem of (tuples of) $G$-bundles with decoration where there is a ``nice" \footnote{Meaning that the semistable locus admits an adequate moduli space.} semistability condition coming from either Mumford's GIT \cite{mumford-git} or $\Theta$-stability  \cite{halpernleistner2018structure, heinloth-hilbertmumford}.

We give two applications of the main theorem. First we explain how to functorially view any stack of decorated $G$-bundles admitting a good moduli space as a global quotient $Y/G$ in characteristic $0$ (Proposition \ref{prop: framed moduli spaces}). If the good moduli space is quasiprojective, then $Y$ is quasiprojective and admits an ample $G$-linearization, and so it follows that in this context any quasiprojective moduli space built using stack-theoretic techniques arises from a GIT problem. Secondly, we show that the moduli stack of semistable meromorphic $G$-Higgs bundles on a family of smooth curves in characteristic $0$ is smooth (Proposition \ref{prop: smoothness semistable meromorphic Higgs bundles}). This implies that the corresponding moduli space is flat in families, and the fibers have klt singularities.

We also explain how to apply the same techniques to Gieseker semistable $G$-bundles (in the sense of \cite{schmitt.singular} \cite{glss.singular.char} \cite{rho-sheaves-paper}) for higher dimensional $X$.

\textbf{Acknowledgements:} I would like to thank Mark Andrea de Cataldo, Tom\'as L. G\'omez, Daniel Halpern-Leistner and Nicolas Templier for helpful discussions. I would also like to thank anonymous referees for useful comments on the manuscript.
\end{section} 

\begin{section}{Some notation and setup}
We work over a field $k$ of arbitrary characteristic. Let $X$ be a geometrically reduced and geometrically connected projective scheme over $k$. Let $G$ be a reductive algebraic group over $k$. We denote by $\mathrm{Bun}_{G}(X)$ the moduli stack of $G$-bundles on $X$. This is an algebraic stack locally of finite type over $k$ and with affine diagonal, by applying \cite[Thm. 1.2]{hall-rydh-tannaka} with source $Z=X$ and target  the classifying stack $BG$.

Let $S$ be an affine $k$-scheme, and let $E$ be a $G$-bundle on $X \times S$ (we will be mainly interested in the case when $S = \mathrm{Spec}(k)$). For any $k$-scheme $Y$ equipped with a $G$-action, we denote by $E \times^{G} Y$ the associated fiber bundle over $X \times S$ with fibers isomorphic to $Y$. In the case when $Y= G$ equipped with the conjugation action, the associated fiber bundle $E \times^{G, \mathrm{Ad}} G$ is a reductive group scheme over $X \times S$, which we will denote by $\underline{Aut}(E)$. We denote by $\mathrm{Aut}(E)$ the contravariant functor from $S$-schemes into sets that sends an $S$-scheme $R$ into the set of sections $s: X \times R \to \underline{Aut}(E)\times_{S} R$ of the natural morphism $\underline{Aut}(E)\times_{S} R \to X \times R$. This functor is represented by an affine algebraic group over $S$, which is the automorphism group of $E$ when viewed as a $S$-point of the stack $\mathrm{Bun}_{G}(X)$.

For every $S$-point $x \in X(S)$, we denote by $\underline{Aut}(E)|_x \vcentcolon = \underline{Aut}(E) \times_{(X \times S,x)} S$ the fiber over $x$ of the group scheme $\underline{Aut}(E)$. We define a homomorphism $\mathrm{res}_x: \mathrm{Aut}(E) \to \underline{Aut}(E)|_x$ of group schemes at the level of functors by sending a section of $\underline{Aut}(E) \to X \times S$ to its restriction at $x$.

We will be interested in $G$-bundles in $\mathrm{Bun}_{G}(X)$ equipped with some type of decoration. One way to make this type of moduli problem precise is to consider a quasiseparated algebraic stack $\mathcal{M}$ that is locally of finite type over $k$, and is equipped with a representable morphism $\varphi: \mathcal{M} \to \mathrm{Bun}_{G}(X)$. In order to expect some rigidity properties for the automorphism groups of points, we need to impose some type of semistability assumptions on the ``$G$-bundles with decoration" that we are considering. This will be encoded by the existence of an adequate moduli space in the sense of \cite{alper_adequate}. We include the definition here for the sake of completeness.
\begin{defn}[\cite{alper_adequate}]
An algebraic stack $\mathcal{M}$ is said to admit an adequate moduli space if there is a morphism $\phi: \mathcal{M} \to Y$ to an algebraic space $Y$ such that the following properties are satisfied.
\begin{enumerate}[(1)]
    \item Let $p: U = \mathrm{Spec}(A) \to Y$ \'etale morphism, and let $\mathcal{A} \to \mathcal{B}$ be surjection of quasicoherent $\mathcal{O}_{\mathcal{M}}$-algebras. Then, for any section $t \in \Gamma(U, p^* \phi_*(\mathcal{B}))$, there exists some $N>0$ and a section $s \in \Gamma(U, p^*\phi_*(\mathcal{A}))$ such that $s$ maps to $t^N$.
    \item The natural morphism $\mathcal{O}_{Y} \to \phi_* \mathcal{O}_{\mathcal{M}}$ is an isomorphism.
\end{enumerate}
\end{defn}
Even though this notion provides a very convenient general setting for our theorem, we should stress that it is not necessary to understand the technicalities of adequate moduli spaces for our argument. The following are the main facts from \cite{alper_adequate} that the reader needs to know for our purposes.
\begin{enumerate}[(A)]
    \item If $\mathcal{M}$ admits an adequate moduli space, then the automorphism group of any closed geometric point of the stack $\mathcal{M}$ is geometrically reductive \cite[Prop. 9.3.4]{alper_adequate}. Moreover, an affine geometrically reductive group does not admit nontrivial normal connected smooth unipotent algebraic subgroups (such smooth unipotent subgroup is geometrically reductive by \cite[Thm. 9.4.1]{alper_adequate}, and so it is reductive by \cite[Thm. 9.7.5]{alper_adequate}, thus forcing it to be trivial).
    \item If the moduli problem parameterized by the points of $\mathcal{M}$ admits a moduli space construction via Mumford's GIT \cite{mumford-git}, then $\mathcal{M}$ admits an adequate moduli space \cite[Thm. 9.1.4]{alper_adequate}.
\end{enumerate}

We shall assume that our stack $\mathcal{M}$ admits an adequate moduli space. Let $S$ be a $k$-scheme. Since the morphism $\varphi: \mathcal{M} \to \mathrm{Bun}_{G}(X)$ is representable and $\mathcal{M}$ is quasiseparated, for any $S$-point $p \in \mathcal{M}(S)$ the natural homomorphism $\mathrm{Aut}(p) \to \mathrm{Aut}(\varphi(p))$ of automorphism group algebraic spaces is a monomorphism of finite type, and therefore it is separated and quasi-finite \cite[\href{https://stacks.math.columbia.edu/tag/0463}{Tag 0463}]{stacks-project}. Since $\mathrm{Aut}(\varphi(p))$ is a scheme, this implies that $\mathrm{Aut}(p)$ is actually a group scheme \cite[\href{https://stacks.math.columbia.edu/tag/03XX}{Tag 03XX}]{stacks-project}. For any $S$-point $x \in X(S)$, we will abuse notation and also denote by $\mathrm{res}_x$ the composition of group scheme homomorphisms $\mathrm{Aut}(p) \hookrightarrow \mathrm{Aut}(\varphi(p)) \xrightarrow{\mathrm{res}_x} \underline{Aut}(\varphi(p))|_x$.
\end{section}
\begin{section}{The main theorem}
For our main theorem, recall that $k$ is an arbitrary field, $X$ is a geometrically reduced and geometrically connected projective scheme over $k$, and $G$ is a reductive group over $k$.
\begin{thm} \label{thm: main theorem}
Let $\mathcal{M}$ be a quasiseparated finite type algebraic stack admitting an adequate moduli space, and equipped with a representable morphism $\varphi: \mathcal{M} \to \mathrm{Bun}_{G}(X)$. Let $p \in \mathcal{M}(k)$. Then, for all $x \in X(k)$, the kernel of $\mathrm{res}_x: \mathrm{Aut}(p) \to \underline{Aut}(\varphi(p))|_x$ is a finite group scheme with unipotent geometric points. In particular $\mathrm{res}_x$ is a closed immersion if $\mathrm{char}(k) = 0$.
\end{thm}

\begin{remark}
    For the proof of \Cref{thm: main theorem}, we are actually only using the following geometric properties of the stack $\mathcal{M} \to \Bun_{G}(X)$:
    \begin{enumerate}[(1)]
        \item Every point of $\mathcal{M}$ specializes to a closed point.
        \item The stabilizer of every closed point is a reductive group.
    \end{enumerate}
    For (1) we just need to know that $\mathcal{M}$ is quasicompact and quasiseparated, because then the topological space of $\mathcal{M}$ is spectral \cite[\href{https://stacks.math.columbia.edu/tag/0DQN}{Tag 0DQN}]{stacks-project}. On the other hand (2) is implied by $\mathrm S$-completeness of the stack $\mathcal{M}$ \cite[Prop. 3.47]{alper2019existence}. In particular \Cref{thm: main theorem} provides a necessary condition for the $\mathrm S$-completeness of a quasiseparated finite type stack $\mathcal{M}$ with a representable morphism $\mathcal{M} \to \Bun_{G}(X)$.
\end{remark}

In order to prove Theorem \ref{thm: main theorem}, we will, without loss of generality, pass to the algebraic closure of $k$ and assume that $k$ is algebraically closed for the rest of this section. Let us denote by $K_x$ the kernel of the homomorphism of group schemes $\mathrm{res}_x$. Our goal is to show that $K_x$ has unipotent $k$-points and is finite over $k$. We start by proving the former.
\begin{lemma} \label{lemma: unipotent points}
Any $k$-point of the algebraic group $K_x$ is unipotent.
\end{lemma}
\begin{proof}
We shall denote by $E$ the $G$-bundle on $X$ corresponding to $\varphi(p)$. Let $g$ be a $k$-point of $K_x$. There are natural immersions of algebraic groups $K_x \hookrightarrow \mathrm{Aut}(p) \hookrightarrow \mathrm{Aut}(E)$. We can without loss of generality replace $g$ with its image in the algebraic group $\mathrm{Aut}(E)$ of automorphisms of the $G$-bundle $E$. We shall think of $g$ as a section
\[ g: X \to \underline{Aut}(E) = E \times^{G, \mathrm{Ad}} G\]
Recall that the group $G$ acts on itself via conjugation; let us denote by $k[G]^G$ the subring of invariants of the coordinate ring of $G$. We set $G/\!/G \vcentcolon = \mathrm{Spec}(k[G]^G)$. The inclusion $k[G]^G \subset k[G]$ induces a $G$-equivariant morphism $G \to G/\!/G$, where $G$ acts on the right-hand side via the trivial action. This induces a morphism of fiber bundles $E \times^{G, \mathrm{Ad}}G \to E \times^{G} (G/\!/G) = X \times (G/\!/G)$. Let us consider the composition of morphisms of $X$-schemes
\[X \xrightarrow{g} E \times^{G, \mathrm{Ad}} G \to X \times (G/\!/G) \]
This composition amounts to a morphism $f: X \to G/\!/G$. Since $X$ is proper, geometrically reduced and geometrically connected, we have $H^0(X, \mathcal{O}_X) = k$ \cite[\href{https://stacks.math.columbia.edu/tag/0BUG}{Tag 0BUG}]{stacks-project}. Using that $G/\!/G$ is affine, this implies that the morphism $f$ must factor through a point $X \to \mathrm{Spec}(k) \to G/\!/G$. In order to determine this value, we can look at the image of any point on $X$. But, since by assumption $g \in K_x(k)$, the image of $x$ in $G/\!/G$ agrees with the image of the identity of $G$. Therefore we see that for any $y \in X(k)$ the image of $\mathrm{res}_y(g)$ under the natural morphism $\underline{Aut}(E)|_{y} \to G/\!/G$ agrees with the image of the identity. 

Note that, since $k$ is algebraically closed, $\underline{Aut}(E)|_{y}$ is (noncanonically, up to conjugation) isomorphic to $G$. Since the image of $\mathrm{res}_y(g)$ in $G/\!/G$ is the same as the image of the identity $1_G$, it satisfies that $f(\text{res}_y(g)) = f(1_G)$ for all $f \in k[G]^G$, and so $\mathrm{res}_y(g)$ is unipotent for all $y$ by \cite[Cor. 6.7]{steinberg-regular}. We shall show that this implies that $g \in \mathrm{Aut}(E)(k)$ is itself unipotent. Fix a faithful representation $G \hookrightarrow \mathrm{GL}_n$ of the linear algebraic group $G$. We will denote by $\mathcal{V}$ the associated $\mathrm{GL}_n$-bundle $E\times^G\mathrm{GL}_n$, which we can also view as a vector bundle of rank $n$ on $X$. There is a closed immersion of group schemes $E \times^{G, \mathrm{Ad}} G \hookrightarrow \mathcal{V} \times^{\mathrm{GL}_n, \mathrm{Ad}} \mathrm{GL}_n$, which induces an immersion of algebraic groups $\psi: \mathrm{Aut}(E) \hookrightarrow \mathrm{Aut}(\mathcal{V})$. Since the restriction of $g$ to each fiber $E|_y$ is unipotent, it follows that the restriction of its image $\psi(g)$ to each fiber is unipotent. This implies that the endomorphism $(\psi(g)- \mathrm{Id}_{\mathcal{V}})^n$ vanishes on every fiber, and so $(\psi(g)- \mathrm{Id}_{\mathcal{V}})^n \in \mathrm{End}(\mathcal{V})$ is identically $0$, since $X$ is reduced. Note that $\mathrm{Aut}(\mathcal{V})$ acts faithfully via multiplication on the finite $k$-vector space $\mathrm{End}(\mathcal{V})$. The equation $(\psi(g)- \mathrm{Id}_{\mathcal{V}})^n=0$ shows that $g$ acts unipotently on $\mathrm{End}(\mathcal{V})$ and thus we conclude that $g$ is indeed a unipotent element.
\end{proof}

\begin{proof}[Proof of Theorem \ref{thm: main theorem}]
After Lemma \ref{lemma: unipotent points}, we are only left to show that $K_x$ is finite. For this it suffices to prove that the dimension of $K_x$ is $0$, because $K_x$ is of finite type. We will show $\mathrm{dim}(K_x)=0$ by degenerating to a closed point of the stack $\mathcal{M}$. 

Let $\mathfrak{Z}$ denote the scheme-theoretic closure of the morphism $p: \mathrm{Spec}(k) \to \mathcal{M}$ \cite[\href{https://stacks.math.columbia.edu/tag/0CMH}{Tag 0CMH}]{stacks-project}. We have that $\mathfrak{Z}$ is of finite type and quasiseparated over $k$, and so it contains a closed point $z$ by \cite[\href{https://stacks.math.columbia.edu/tag/0DQN}{Tag 0DQN}]{stacks-project}. Let $U \to \mathfrak{Z}$ be a neighborhood of $z$ in the smooth topology, with $U$ a scheme of finite type over $k$. By construction, there exists a dense open subset of $U$ whose points map to $p$ under the composition $U \to \mathfrak{Z} \to \mathcal{M}$. Choose a closed point $u \in U(k)$ that maps to $z$. After further replacing $U$ with an affine neighborhood of $u$ inside an irreducible component, we obtain an affine integral scheme of finite type $U$ with a morphism $\widetilde{p}: U \to \mathcal{M}$ such that every point in a dense open subset $V \subset U$ is sent to $p$, and at least one point $u \in U(k)$ is sent to the closed point $z$ of $\mathcal{M}$. This is the degeneration we will use.

Consider the group scheme $\mathrm{Aut}(\widetilde{p}) \to U$. By construction, we know that the restriction of $\mathrm{Aut}(\widetilde{p})$ to every point in $V \subset U$ is isomorphic to the original automorphism group $\mathrm{Aut}(p)$. On the other hand, since $u$ maps a closed point of the stack $\mathcal{M}$ that admits an adequate moduli space, $\mathrm{Aut}(\widetilde{p})|_u$ is geometrically reductive.

For any $x \in X(k)$ we can consider the restriction morphism of $U$-group spaces $\mathrm{res}_x: \mathrm{Aut}(\widetilde{p}) \to \mathrm{Aut}(\varphi(\widetilde{p})) \to \underline{Aut}(\varphi(\widetilde{p}))|_{x \times U}$, where we define everything relative to $U$ using the section $x \times U \to X \times U$. We denote by $\widetilde{K}_x$ the group space of finite type over $U$ given by the kernel of $\mathrm{res}_x$. By definition the fibers of $\widetilde{K}_x$ over the open dense $V \subset U$ recover our original kernel $K_x$.

We would like to show that the constant fiber dimension $\mathrm{dim}(K_x)$ of $\widetilde{K_x}$ over $V \subset U$ is equal to $0$. By upper semicontinuity on the base $U$ for the dimension of the fibers of the group scheme $\widetilde{K}_x$ (\cite[Prop 4.1]{sga3} and \cite[Cor. 2.4.1]{sga3}), it suffices to show that the fiber dimension over the closed point $u$ is $0$.


In summary, we are allowed to replace the point $p$ with its degeneration $u$. In particular, we can assume without loss of generality that $p$ is a closed point of the stack $\mathcal{M}$, and so $\mathrm{Aut}(p)$ is geometrically reductive. The kernel $K_x \subset \mathrm{Aut}(p)$ is a normal subgroup scheme of $\mathrm{Aut}(p)$. The (smooth) reduced subgroup of the neutral component $\mathrm{Red}(K_x)^0 \subset (K_x)^0$ is also normal in $\mathrm{Aut}(p)$. Since every $k$-point of $\mathrm{Red}(K_x)^0$ is unipotent, we conclude that $\mathrm{Red}(K_x)^0$ is a unipotent algebraic group. Since $\mathrm{Aut}(p)$ is an affine geometrically reductive group scheme, it does not admit nontrivial smooth connected normal unipotent subgroups (cf. Fact (A) above). Therefore, $\mathrm{Red}(K_x)^0$ is trivial. We conclude that $K_x$ is finite, and so its dimension is $0$ as desired.

Finally, if the characteristic of $k$ is $0$, then the group $K_x$ is smooth with unipotent geometric points. In this case any non-identity geometric point $g$ of $K_x$ has infinite order, contradicting the finiteness of $K_x$. It follows that $K_x$ must the the trivial algebraic group, and hence $\mathrm{res}_x$ is a closed immersion.

\end{proof}

\begin{remark} \label{remark: general base}
The same argument applies verbatim in families. Suppose that $S$ is a Noetherian $k$-scheme, $X \to S$ is a flat projective morphism with connected and reduced geometric fibers, and $G \to S$ is a reductive group scheme. There is a relative stack $\Bun_{G}(X/S) \to S$ that classifies $G$-bundles on the fibers of the morphism $X \to S$. By \cite[Thm. 1.2]{hall-rydh-tannaka}, $\Bun_{G}(X/S) = \mathrm{Map}_S(X, S/G_S)$ is algebraic and locally of finite type over $S$. Let $\mathcal{M} \to \Bun_{G}(X/S)$ be a representable morphism, where $\mathcal{M}$ is quasiseparated and of finite type over $S$. If $\mathcal{M}$ admits an adequate moduli space, then the same rigidity statement as in theorem \ref{thm: main theorem} applies for the automorphisms of any field-valued point of $\mathcal{M}$.
\end{remark}
\end{section}
\begin{section}{Examples over a curve}
We will focus on the case when $X=C$ is a smooth projective curve. First, in order to fix ideas, we present some simple examples of automorphism groups and restriction morphisms.
\begin{example}
Let $C$ be an elliptic curve over $k$, and set $G = \mathrm{GL}_2$. The stack $\Bun_{\mathrm{GL}_2}(C)$ parametrizes rank 2 vector bundles on $C$. Recall that a vector bundle $\mathcal{E}$ of rank 2 on $C$ is called semistable if for all line subbundles $\mathcal{L} \subset \mathcal{E}$ we have the inequality of degrees $\mathrm{deg}(\mathcal{L}) \leq \frac{1}{2} \mathrm{deg}(E)$. Let $p, x \in C(k)$ be two $k$-points of the curve. Let us describe some instances of the restriction morphism.
\begin{enumerate}[(1)]
    \item Since $H^1(C, \mathcal{O}_C) = k$, there is a unique nonsplit extension up to isomorphism
    \[ 0 \to \mathcal{O}_C \to \mathcal{E} \to \mathcal{O}_C \to 0\]
    The vector bundle $\mathcal{E}$ is semistable. Any point of the automorphism group scheme $\mathrm{Aut}(\mathcal{E})$ preserves the subsheaf $\mathcal{O}_C \subset \mathcal{E}$. Therefore we can define a natural homomorphism of group schemes $\mathrm{Aut}(\mathcal{E}) \to \mathrm{Aut}(\mathcal{O}_C) \cong \mathbb{G}_m$ given by restricting the action to $\mathcal{O}_C \subset \mathcal{E}$. This homomorphism admits a splitting given by the scaling action of $\mathbb{G}_m$ on $\mathcal{E}$. The kernel of $\mathrm{Aut}(\mathcal{E}) \to \mathbb{G}_m$ is isomorphic to $\mathrm{Hom}(\mathcal{O}_C, \mathcal{E}) \cong \mathrm{Hom}(\mathcal{O}_C, \mathcal{O}_C) \cong \mathbb{G}_a$. This induces an identification $\mathrm{Aut}(\mathcal{E}) \cong \mathbb{G}_m \times \mathbb{G}_a$, which we can use to embed $\mathrm{Aut}(\mathcal{E})$ into $\mathrm{GL}_2$
    \[ i: \mathrm{Aut}(\mathcal{E}) \cong \mathbb{G}_m \times \mathbb{G}_a \hookrightarrow \mathrm{GL}_2, \; \; \; \; \; \; \; (t,a) \mapsto \begin{bmatrix} t & a\\ 0 & t\end{bmatrix}\]
    The restriction of the short exact sequence $0 \to \mathcal{O}_C \to \mathcal{E} \to \mathcal{O}_C \to 0$ to the point $x$ splits, as it is a short exact sequence of $k$-vector spaces. Any choice of splitting induces a trivialization of the fiber $\mathcal{E}_x$, and therefore an isomorphism $\mathrm{Aut}(\mathcal{E}_x) \cong \mathrm{GL}_2$. Under this isomorphism, the restriction morphism $\mathrm{res}_x$ is identified with the inclusion $i: \mathrm{Aut}(\mathcal{E}) \to \mathrm{GL}_2 \cong \mathrm{Aut}(\mathcal{E}_x)$ (this turns out to be independent of the choice of splitting at $x$).
    
    \item Consider a nontrivial extension
    \[ 0 \to \mathcal{O}_C \to \mathcal{E} \to \mathcal{O}_{C}(p) \to 0\]
    The vector bundle $\mathcal{E}$ geometrically stable, hence it is semistable and its group of automorphisms $\mathrm{Aut}(E)$ is the multiplicative group $\mathbb{G}_m$ given by scaling by the constants. The restriction morphism $\mathrm{res}_x: \mathrm{Aut}(\mathcal{E}) \hookrightarrow \mathrm{Aut}(\mathcal{E}_x)$ is the natural inclusion of the subgroup of scalars $\mathbb{G}_m$ inside $\mathrm{Aut}(\mathcal{E}_x)$.
    
    \item Consider the direct sum $\mathcal{E} = \mathcal{O}_C \oplus \mathcal{O}_C(p)$. In this case $\mathcal{E}$ is not semistable. The automorphisms of $\mathcal{E}$ are the group of upper triangular matrices
    \[ \mathrm{Aut}(\mathcal{E}) = \begin{bmatrix} \mathrm{Aut}(\mathcal{O}_C) & \mathrm{Hom}(\mathcal{O}_C, \mathcal{O}_C(p)) \\ 0 & \mathrm{Aut}(\mathcal{O}_C(p)) \end{bmatrix} =  \begin{bmatrix} \mathbb{G}_m & \mathbb{G}_a \\ 0 & \mathbb{G}_m \end{bmatrix}\]
    For any $x$, the direct sum decomposition induces a trivialization $\mathcal{E}_x \cong k \oplus k$, and therefore an identification $\mathrm{Aut}(\mathcal{E}_x) \cong \mathrm{GL}_2$. If $x \neq p$, then the restriction morphism $\mathrm{res}_x: \mathrm{Aut}(\mathcal{E}) \to \mathrm{Aut}(\mathcal{E}_x) \cong \mathrm{GL}_2$ is the closed immersion of the subgroup upper triangular matrices inside $\mathrm{GL}_2$. However, if $x =p$, then the restriction morphism is not an immersion, and is given by
    \[\mathrm{res}_x: \begin{bmatrix} \mathbb{G}_m & \mathbb{G}_a \\ 0 & \mathbb{G}_m \end{bmatrix} \to \mathrm{GL}_2, \; \; \; \; \begin{bmatrix} a & b \\ 0 & c \end{bmatrix} \mapsto \begin{bmatrix} a & 0 \\ 0 & c \end{bmatrix}\]
\end{enumerate}
\end{example}

We shall now present some examples to which the general result in Theorem \ref{thm: main theorem} applies. 
\begin{example}[Semistable $G$-bundles on $C$]
Ramanathan \cite{ramanathan-stable} defined a notion of semistability for $G$-bundles, which determines an open substack $\mathcal{M}$ of $\mathrm{Bun}_{G}(C)$. A $G$-bundle $E$ is called semistable if for all reductions of structure group $E_P$ to a parabolic subgroup $P \subset G$ and all $P$-dominant characters $\chi: P \to \mathbb{G}_m$, the associated line bundle $E_{P} \times^{P} \chi$ on $C$ has nonpositive degree. Moreover, Ramanathan constructed a moduli space for $\mathcal{M}$ using GIT when $G$ is connected reductive \cite{ramanathan-moduli-git} (see also \cite[Cor. 5.5.3]{glss.large} in the case of positive characteristic). We can apply Theorem \ref{thm: main theorem} to the open substack of semistable $G$-bundles of fixed topological type, which is of finite type, to obtain the rigidity result in this context.

We note that in this case the result in characteristic $0$ also follows from a cohomology computation without the need to degenerate to a polystable object. Indeed, by deformation theory we have $\text{Lie}(\mathrm{Aut}(E)) = H^0(C, E \times^{G, \mathrm{Ad}} \mathfrak{g}))$, where $\mathfrak{g}$ is the Lie algebra of $G$ equipped with its adjoint action. The restriction $\text{Lie}(\text{res}_x): \text{Lie}(\mathrm{Aut}(E)) \to \text{Lie}(\mathrm{Aut}(E|_x))$ is identified with the restriction morphism on global sections
\[ H^0(C, (E \times^{G, \mathrm{Ad}} \mathfrak{g})) \to H^0(C, (E \times^{G, \mathrm{Ad}} \mathfrak{g})|_x)\]
Therefore the Lie algebra of the kernel $K_x$ is given in this case by the kernel of the above morphism
\[H^0(C, (E \times^{G, \mathrm{Ad}} \mathfrak{g})\otimes \mathcal{O}_C(-x)) = \mathrm{Hom}(\mathcal{O}_C(x), \, E \times^{G, \mathrm{Ad}} \mathfrak{g})\]
By \cite[Thm. 3.18]{rr-instability-flag}, the adjoint vector bundle $E \times^{G, \mathrm{Ad}} \mathfrak{g}$ is semistable (and of degree $0$), so it does not admit any homomorphisms from the (semistable) line bundle $\mathcal{O}_C(x)$ of positive degree. This fact, along with the unipotence established in Lemma \ref{lemma: unipotent points}, imply that $K_x$ is trivial.

We also note that the triviality of the Lie algebra of $K_x$ follows from the same argument in positive characteristic as long as $\mathrm{char}(k)$ is larger than the height of the adjoint representation $\mathfrak{g}$, by using \cite[Prop. 6.9]{biswas-holla-hnreduction}.
\end{example}

\begin{example}[Semistable Hitchin pairs on $C$] \label{example: semistable Hitchin pairs}
Let us denote by $\mathfrak{g}$ the Lie algebra of $G$, equipped with its adjoint action by $G$. We fix a line bundle $\mathcal{L}$ on $C$. A ($\mathcal{L}$-)twisted Hitchin $G$-bundle is the data of a pair $(E, s)$, where $E$ is a $G$-bundle on $C$ and $s$ is a section of the vector bundle $\left(E \times^{G, \mathrm{Ad}}\mathfrak{g}\right) \otimes \mathcal{L}$. An automorphism of such a twisted Hitchin $G$-bundle is an automorphism of the underlying $G$-bundle $E$ that is compatible with the section $s$. Hitchin pairs are parametrized by an algebraic stack $\mathfrak{X}$ that is affine and of finite type over $\mathrm{Bun}_{G}(C)$. Indeed, let $T \to \Bun_{G}(C)$ be a scheme-valued point corresponding to a $G$-bundle $E$ on $C_T$. The the fiber over $T$ of the forgetful morphism $\mathfrak{X} \to \mathrm{Bun}_G(C)$ parametrizes global sections of the vector bundle $E \times^{G, \mathrm{Ad}} \mathfrak{g}$. The functor of such global sections is represented by a relatively affine scheme of finite type over $T$ by applying \cite[\href{https://stacks.math.columbia.edu/tag/08K6}{Tag 08K6}]{stacks-project} with $\mathcal{F} = \mathcal{O}_{C_T}$ and $\mathcal{G} = E \times^{G, \mathrm{Ad}} \mathfrak{g}$. 

There is a notion of stability for Hitchin pairs analogous to the one for $G$-bundles. In the case of Hitchin pairs, we only check the character condition on parabolic reductions that are compatible with the section (see \cite[Def. 4.6]{anchouche-biswas-eintein-hermitian} or \cite[Defn. 3.3]{bruzzo-otero} for concrete explanations of what compatibility entails, they can also be viewed as $\Theta$-filtrations in the sense of \cite{halpernleistner2018structure}). It can be shown (for example using GIT in \cite[2.8]{schmitt-decorated}) that when $G$ is connected reductive each connected component of the open semistable locus $\mathcal{M}$ is of finite type and admits an adequate moduli space. Therefore, Theorem \ref{thm: main theorem} can be applied to semistable Hitchin pairs.
\end{example}

\begin{example}[Semistable parabolic $G$-bundles] Fix finitely many $k$-points \\$\{p_1, p_2, \ldots, p_n\}$ of the curve $C$. For each index $1 \leq i \leq n$, let $P_i \subset G$ be a choice of parabolic subgroup. The data $\{p_i\}$, $\{P_i\}$ defines a stack $\mathfrak{X}$ of $G$-bundles with parabolic structure. More precisely, $\mathfrak{X}$ parametrizes tuples $(E, s_1, \ldots, s_n)$, where $E$ is a $G$-bundle on $C$ and $s_i$ is a reduction of structure group of the fiber $E|_{p_i}$ to the parabolic subgroup $P_i$. The forgetful morphism $\mathfrak{X} \to \Bun_{G}(C)$ is representable and proper, since the fibers parametrize parabolic reductions of the $G$-bundle at finitely many points of the curve, and so they are isomorphic \'etale locally on the source to the product of flag varieties $\prod_i G/P_i$. Heinloth and Schmitt \cite{heinloth-schmitt} defined notions of semistability depending on certain admissible parameters associated to the parabolics $P_i$. Such stability condition determines an open substack $\mathcal{M}$ of the moduli stack of $G$-bundles with parabolic structure. Using GIT, it is shown in \cite{heinloth-schmitt} that $\mathcal{M}$ admits an adequate moduli space when $G$ is connected reductive, and so Theorem \ref{thm: main theorem} can be applied in this situation. Note that this theorem applies to any restriction $\mathrm{res}_x$, regardless of whether $x$ coincides with one of the points $p_i$.
\end{example}

\begin{example}[Holomorphic chains] \label{examples: holomorphic pairs}
We can take $G = \prod_i \mathrm{GL}_{n_i}$ to be a finite product of general linear groups. The moduli stack $\mathfrak{X}$ of chains of vector bundles on $C$ such that the $i^{th}$ bundle has rank $n_i$ admits an affine morphism of finite type $\mathfrak{X} \to \prod_i \Bun_{\mathrm{GL}_{n_i}}(C) = \Bun_{G}(C)$. Indeed, for any scheme-valued point $T \to \Bun_{G}(C)$ corresponding to a tuple $(\mathcal{E}_i)_i$ of vector bundles on $C_T$, the $T$-fiber of the forgetful morphism $\mathfrak{X} \to \Bun_{G}(C)$ classifies global sections of the vector bundle $\bigoplus_i \mathcal{H}\mathrm{om}(\mathcal{E}_i, \mathcal{E}_{i+1})$. The functor of such global sections is represented by a relatively affine scheme of finite type over $T$ \cite[\href{https://stacks.math.columbia.edu/tag/08K6}{Tag 08K6}]{stacks-project}. There is a family of stability conditions for this moduli problem considered in \cite{consul-prada-chains}, \cite[2.1]{geometry-hol-chains}. We can let $\mathcal{M}$ be the open substack of semistable chains. The existence of adequate moduli spaces for this stack was proven via GIT in \cite{schmitt-quiver1, schmitt-quiver2}, and therefore our theorem also applies in this context.
\end{example}

\begin{example}[$G$-bundles with a section] \label{example: bundle with a section}
Suppose that $k$ has characteristic $0$, and let $V$ be a linear representation of $G$. The stack of $G$-bundles with a $V$-section parameterizes pairs $(E, s)$, where $E$ is a $G$-bundle on $C$ and $s$ is a section of the vector bundle $E \times^{G} V$. The forgetful morphism to the stack $\Bun_{G}(C)$ is affine and of finite type, by a similar argument as in Examples \ref{example: semistable Hitchin pairs} and \ref{example: bundle with a section} using \cite[\href{https://stacks.math.columbia.edu/tag/08K6}{Tag 08K6}]{stacks-project}. There are natural semistability conditions for this moduli problem, and the moduli space for such pairs has been constructed in by GIT in \cite[2.8]{schmitt-decorated}. Therefore Theorem \ref{thm: main theorem} applies in this case to show that the subgroup of automorphisms of the $G$-bundle $E$ that preserve the section embeds via restriction into the automorphism of any fixed fiber.

Similar considerations apply to the moduli of coherent systems as in \cite{le-potier-coherent, schmitt-coherent-systems} and the related moduli of (generalized) pairs considered in \cite{torsion-freepaper} (in the special case when the underlying variety $X$ is a curve).
\end{example}

\end{section}
\begin{section}{Applications}
For our applications, we will make use of the following notions. We include the definitions for the convenience of the reader.
\begin{defn}
    A quasicompact and quasiseparated morphism $f: \mathfrak{X} \to \mathfrak{Y}$ of algebraic stacks is called cohomologically affine if the pushforward functor $f_*(-): \mathrm{QCoh}(\mathfrak{X}) \to \mathrm{QCoh}(\mathfrak{Y})$ between the abelian categories of quasicoherent sheaves is exact. 
\end{defn}

\begin{defn}
    Let $\mathfrak{X}$ be an algebraic stack. An algebraic space $M$ equipped with a quasicompact and quasiseparated morphism $\varphi: \mathfrak{X} \to M$ is called a good moduli space if
    \begin{enumerate}[(1)]
        \item $\varphi$ is cohomologically affine.
        \item $\varphi_*(\mathcal{O}_{\mathfrak{X}}) = \mathcal{O}_M$.
    \end{enumerate}
\end{defn}
\begin{remark}
    The notion of adequate moduli space is weaker than that of good moduli space. In particular every good moduli space is also an adequate moduli space. Both notions agree in characteristic $0$, but differ in positive and mixed characteristic.
\end{remark}
Our first application is a functorial presentation of stacks of decorated $G$-bundles admitting a good moduli space as global quotient stacks. For this we use the parameter space of bundles framed at a point, as in \cite[\S9]{simpson-repnII}.
\begin{prop} \label{prop: framed moduli spaces}
Suppose that $\mathrm{char}(k) = 0$, and let $\mathcal{M}$ be a quasiseparated finite type algebraic stack admitting a good moduli space $\mathcal{M} \to M$, and equipped with a representable morphism $\varphi: \mathcal{M} \to {\mathrm{Bun}}_{G}(X)$. 
\begin{enumerate}[(1)]
    \item If $X(k) \neq \emptyset$, then there exists an algebraic space $Y \to M$ affine and of finite type over $M$ and an action of $G$ on $Y$ such that $\mathcal{M} \cong Y/G$. Furthermore, the construction of $Y$ is functorial, in the sense that every morphism of stacks $\mathcal{M}_1 \to \mathcal{M}_2$ over $\mathrm{Bun}_{G}(X)$ induces a $G$-equivariant morphism of the corresponding algebraic spaces $Y_1 \to Y_2$.
    \item If the good moduli space $M$ is quasiprojective, then $Y$ is quasiprojective and admits a $G$-linearized ample line bundle. In this case $M$ is the GIT quotient of the linearized $G$-scheme $Y$, in the sense of Mumford.
\end{enumerate} 
\end{prop}
\begin{proof}
\begin{enumerate}[(1)]
\item Choose $x \in X(k)$. Consider the morphism
\[  \Bun_{G}(X) \to BG, \; \; \; E \mapsto E|_x\]
We can pull back the universal $G$-bundle $\mathrm{Spec}(k) \to BG$ to obtain the $G$-bundle $\Bun_{G}(X)^{\mathrm{fr}} \to \Bun_{G}(X)$. By definition $\Bun_{G}(X)^{\mathrm{fr}}$ is the stack classifying $G$-bundles on $X$ along with trivializations of the fiber at $x \in X$. Set $Y$ to be the fiber product $Y = \mathcal{M} \times_{\Bun_{G}(X)}\Bun_{G}(X)^{\mathrm{fr}}$. A point in $Y$ corresponds to pair $(p, \psi)$, where $p$ is a point in $\mathcal{M}$ and $\psi$ is a trivialization of the $x$-fiber $E|_x$ of the corresponding $G$-bundle $E= \varphi(p)$. An automorphism of such a point $(p, \psi)$ in $Y$ amounts to an automorphism of $p$ that preserves the trivialization, and so it restricts to the trivial automorphism of $E|_x$. In other words, the algebraic group of automorphisms of $(p, \psi)$ is the kernel $K_x$ of the restriction morphism $\mathrm{res}_x: \mathrm{Aut}(p) \to \mathrm{Aut}(E)|_x$. Theorem \ref{thm: main theorem} for $\mathrm{char}(k) = 0$ applied to $\varphi: \mathcal{M} \to \Bun_{G}(X)$ implies that $K_x$ is trivial. Therefore every point of $Y$ has trivial automorphism group, and so $Y$ is an algebraic space. By construction $Y \to \mathcal{M}$ is a $G$-bundle, so $Y$ admits an action of $G$ (given by changing the trivialization $\psi$), and we have $\mathcal{M} = Y/G$. The fact that $Y$ is of finite type follows because it is a $G$-bundle over the finite type stack $\mathcal{M}$. Since both $Y \to \mathcal{M}$ and the good moduli space are cohomologically affine, so is their composition $Y \to M$. Therefore, the morphism $Y \to M$ of algebraic spaces is affine \cite[Prop. 3.3]{alper-good-moduli}.

The functoriality of $Y$ follows directly from its construction as $\mathcal{M} \times_{\Bun_{G}(X)}\Bun_{G}(X)^{\mathrm{fr}}$; indeed any morphism of stacks $\mathcal{M}_1 \to \mathcal{M}_2$ over $\Bun_{G}(X)$ induces a $G$-equivariant morphism of fiber products $\mathcal{M}_1 \times_{\Bun_{G}(X)}\Bun_{G}(X)^{\mathrm{fr}} \to \mathcal{M}_2 \times_{\Bun_{G}(X)}\Bun_{G}(X)^{\mathrm{fr}}$.

\item If $M$ admits an ample line bundle $L$, then it pulls back to an $G$-linearized ample line bundle $\widetilde{L}$ on $Y$, because $Y \to M$ is affine \cite[\href{https://stacks.math.columbia.edu/tag/0892}{Tag 0892}]{stacks-project}. By the ampleness of $L$, the scheme $M$ is covered by affine open neighborhoods that are the complements of sections of positive powers of $L$. By pulling back these sections to $Y$ and using that $Y \to M$ is affine, we see that $Y$ is covered by affine neighborhoods that are complements of $G$-invariant sections of powers of $\widetilde{L}$. This means that every point of $Y$ is semistable for the line bundle $\widetilde{L}$. It follows that $M$ is the GIT quotient of the linearized quasiprojective $G$-scheme $Y$ (cf. \cite[Remark 13.7]{alper-good-moduli}).
\end{enumerate} 
\end{proof}

For our second application we will again assume that $\mathrm{char}(k) = 0$. Fix a $k$-scheme $S$ of finite type over $k$, and a smooth projective morphism $C \to S$ such that the fibers are geometrically connected curves. Let $D$ be a relative Cartier divisor $D \hookrightarrow C$ such that for all points $s \in S$ the fiber $D_s \hookrightarrow C_s$ is not empty. We shall use the same notation as in Example \ref{example: semistable Hitchin pairs}. Set $\mathcal{L} = \Omega^1_{C/S}(D)$, which is a line bundle on $C$. There is a relative Hitchin stack $\mathrm{Higgs}^D_{G}(C/S) \to S$ of meromorphic $G$-Higgs bundles that parametrizes $\mathcal{L}$-twisted $G$-Higgs bundles Hitchin pairs $(E, \phi)$ on fibers of $C\to S$. We will denote by $\Bun_{G}(C/S) \to S$ the relative stack of $G$-bundles for the morphism $C \to S$ (cf. Remark \ref{remark: general base}). The natural forgetful morphism
\[ \mathrm{Higgs}^D_{G}(C/S) \to \Bun_{G}(C/S), \; \; \; \; (E, \phi) \mapsto E \]
is affine and of finite type. There is an open substack $\mathrm{Higgs}^D_{G}(C/S)^{\mathrm{ss}} \subset \mathrm{Higgs}^D_{G}(C/S)$ whose points parametrize semistable meromorphic Higgs bundles, in the sense of Example \ref{example: semistable Hitchin pairs}. The open and closed substacks $\mathrm{Higgs}^D_{G}(C/S)^{\mathrm{ss}}_{d} \subset \mathrm{Higgs}^D_{G}(C/S)^{\mathrm{ss}}$ where the underlying $G$-bundle has a fixed topological type $d$ are of finite type over $S$. The substacks $\mathrm{Higgs}^D_{G}(C/S)^{\mathrm{ss}}_{d}$ admit good moduli spaces of finite type over $S$, by the GIT construction of \cite[2.8.1.2]{schmitt-decorated} for a single line bundle $L = \mathcal{L}$ and the adjoint representation $\rho: G \to \text{GL}(\mathfrak{g})$ (we note that the GIT construction in the reference applies similarly in the setting of families). In particular the whole stack $\mathrm{Higgs}^D_{G}(C/S)^{\mathrm{ss}}$ admits a good moduli space locally of finite type over $S$.
\begin{prop} \label{prop: smoothness semistable meromorphic Higgs bundles}
With notation and assumptions as in the paragraph above, the moduli stack $\mathrm{Higgs}^{D}_{G}(C/S)^{\mathrm{ss}}$ of semistable meromorphic $G$-Higgs bundles is smooth over $S$. In particular, the corresponding good moduli space is flat over $S$, and all of its $S$-fibers have klt singularities.
\end{prop}
\begin{proof}
We only need to prove that the moduli stack $\mathrm{Higgs}^{D}_{G}(C/S)^{\mathrm{ss}}$ is smooth; then the claims about the moduli space will follow readily from \cite[Thm. 4.16 (ix)]{alper-good-moduli} and \cite[Thm. 5]{moraga-klt}.

To prove smoothness of $\mathrm{Higgs}^{D}_{G}(C/S)^{\mathrm{ss}}$, we shall show that deformations are unobstructed. Let $F \supset k$ be an algebraically closed field, and let $(E, \phi)$ be an $F$-point of $\mathrm{Higgs}^{D}_{G}(C/S)^{\mathrm{ss}}$, corresponding to a $G$-bundle $E$ on the fiber $C_{F}$ and a section $\phi \in H^0(C_{F}, \mathrm{Ad}(E) \otimes \Omega_{C_{F}/F}^1(D_{F}))$. By \cite[\S2,\S3]{br-hitchin-pairs-deformations}, there is a complex of coherent sheaves on $C_{F}$
\[ \mathcal{C}^{\bullet}_{(E,\phi)} = \left[ \mathrm{Ad}(E) \xrightarrow{\partial} \mathrm{Ad}(E)\otimes\Omega^1_{C_{F}/F}(D_{F})\right]\]
that controls the deformation theory of $(E,\phi)$ relative to $S$. In particular, the obstructions to deforming the pair live in the second hypercohomology group
\[ \mathbb{H}^2(\mathcal{C}^{\bullet}_{(E, \phi)}) = \mathrm{coker}\left( H^1(\mathrm{Ad}(E)) \xrightarrow{H^1(\partial)} H^1(\mathrm{Ad}(E)\otimes\Omega^1_{C_{F}/F}(D_{F})) \right) \]
By the duality explained in \cite[\S4]{br-hitchin-pairs-deformations}, Serre duality induces a canonical identification
\[ \mathbb{H}^2(\mathcal{C}^{\bullet}_{(E, \phi)})^{\vee} \cong \mathrm{ker}\left( H^0(\mathrm{Ad}(E)(-D_{F})) \xrightarrow{H^0(\partial(-D_{F}))} H^1(\mathrm{Ad}(E)\otimes\Omega^1_{C_{F}/F}) \right) \]
On the other hand we have $\mathbb{H}^0(\mathcal{C}^{\bullet}_{(E, \phi)}) \subset H^0(\mathrm{Ad}(E))$ given by
\[ \mathbb{H}^0(\mathcal{C}^{\bullet}_{(E, \phi)}) = \mathrm{ker}\left( H^0(\mathrm{Ad}(E)) \xrightarrow{H^0(\partial)} H^0(\mathrm{Ad}(E)\otimes\Omega^1_{C_{F}/F}(D_{F})) \right) \]
By the duality above we have an identification $\mathbb{H}^2(\mathcal{C}^{\bullet}_{(E, \phi)})^{\vee} \cong \mathbb{H}^0(\mathcal{C}^{\bullet}_{(E, \phi)}) \cap H^0(\mathrm{Ad}(E(-D_{F}))$. In other words, $\mathbb{H}^2(\mathcal{C}^{\bullet}_{(E, \phi)})^{\vee}$ is the subset of those sections in $\mathbb{H}^0(\mathcal{C}^{\bullet}_{(E, \phi)}) \subset H^0(\mathrm{Ad}(E))$ that vanish when restricted to the divisor $D_{F}$. Note that $H^0(\mathrm{Ad}(E))$ is the Lie algebra of the group of automorphisms $\mathrm{Aut}(E)$ of the underlying $G$-bundle $E$. Similarly we have that $\mathbb{H}^0(\mathcal{C}^{\bullet}_{(E, \phi)}) \subset H^0(\mathrm{Ad}(E))$ is the Lie algebra of the subgroup $\mathrm{Aut}(E, \phi) \subset \mathrm{Aut}(E)$ of automorphisms of the pair $(E, \phi)$. An element of $\mathbb{H}^2(\mathcal{C}^{\bullet}_{(E, \phi)})^{\vee}\cong \mathbb{H}^0(\mathcal{C}^{\bullet}_{(E, \phi)}) \cap H^0(\mathrm{Ad}(E(-D_{F}))$ would yield an infinitesimal automorphism $\psi$ in $\mathrm{Aut}(E,\phi)(F[\epsilon]/(\epsilon^2)) \subset \mathrm{Aut}(E)(F[\epsilon]/(\epsilon^2))$ that restricts to the identity over $D_{F}$. For any $F$-point $x \in D_{F}(F)$ in the nonempty divisor $D_F$, it follows that the point $\psi \in \mathrm{Aut}(E,\phi)(F[\epsilon]/(\epsilon^2))$ is in the kernel of
\[ \mathrm{res}_x: \mathrm{Aut}(E, \phi) \to \mathrm{Aut}(E_x)\]
By Theorem \ref{thm: main theorem} (which can be applied in families to each finite type open and closed substack $\mathrm{Higgs}^D_{G}(C/S)^{\mathrm{ss}}_{d}$ by Remark \ref{remark: general base}), the morphism $\mathrm{res}_x$ is a monomorphism, and therefore $\psi = 0$. We conclude that $0 = \mathbb{H}^2(\mathcal{C}^{\bullet}_{(E, \phi)})^{\vee}= \mathbb{H}^2(\mathcal{C}^{\bullet}_{(E, \phi)})$, and hence deformations of $(E,\phi)$ are unobstructed.
\end{proof}
\end{section}
\begin{section}{Gieseker semistable $G$-bundles in higher dimensions}
In this section we explain how to apply a similar degeneration argument as in Theorem \ref{thm: main theorem} in the setting of \cite{rho-sheaves-paper}. Let $X$ be a smooth projective variety over a field $k$ of characteristic $0$, equipped with the choice of a fixed ample polarization $\mathcal{O}(1)$. Let $G$ be a connected reductive group over $k$, and fix a faithful homomorphism $ \rho: G \to \prod_{i=1}^b \mathrm{GL}_{r_i}$ into a product of general linear groups. The article \cite{rho-sheaves-paper} defines a moduli stack $\mathrm{Bun}_{\rho}(X)$ of $\rho$-sheaves. It has an open substack $\mathcal{M} \subset \mathrm{Bun}_{\rho}(X)$ of Gieseker semistable $\rho$-sheaves with respect to the given polarization. In this case $\mathcal{M}$ is shown to admit a good moduli space in the sense of \cite{alper-good-moduli}. Any $G$-bundle on $X$ can be canonically regarded as a $\rho$-sheaf, thus inducing an open immersion $\mathrm{Bun}_{G}(X) \hookrightarrow \mathrm{Bun}_{\rho}(X)$. The $G$-bundles that lie in the open substack $\mathcal{M} \subset \mathrm{Bun}_{\rho}(X)$ are called Gieseker semistable. 

One might wonder whether a similar type of rigidity result holds for the automorphisms of a Gieseker semistable $G$-bundle. In fact the same type of degeneration argument shows the following.
\begin{prop}
Let $E$ be a Gieseker semistable $G$-bundle on $X$. Then, there exists a closed subset $Z \subset X$ of codimension at least $2$ such that for every $k$-point $x$ inside the open complement $X \setminus Z$, the restriction $\mathrm{res}_x: \mathrm{Aut}(E) \to \underline{Aut}(E)|_x$ is an immersion of algebraic groups.
\end{prop}
\begin{proof}
We can base change to the algebraic closure of $k$ in order to assume without loss of generality that $k$ is algebraically closed. By Lemma \ref{lemma: unipotent points}, it suffices to check that the kernel of $\mathrm{res}_x$ has dimension $0$. Using the same argument as in the proof of Theorem \ref{thm: main theorem} for the stack $\mathcal{M} \subset \mathrm{Bun}_{\rho}(X)$, we can degenerate $E$ to a ``Gieseker polystable" $\rho$-sheaf, meaning a closed point $u$ in $\mathcal{M}$. In fact, since $\mathcal{M}$ admits a good moduli space, it is locally reductive, so we can use \cite[Lemma 3.24]{alper2019existence} to assume that the degeneration is a $\mathbb{G}_m$-equivariant family over $\mathbb{A}^1_k$ (i.e. a $\Theta$-degeneration of $E$). This way we obtain a $\rho$-sheaf on $\widetilde{p}: \mathbb{A}^1_k \to \mathcal{M}$ that restricts to $E$ on the open $\mathbb{A}^1_k \setminus 0$ and sends the origin to a closed point of $\mathcal{M}$. Recall from \cite{rho-sheaves-paper} that $\widetilde{p}$ amounts to the data of a pair $(\widetilde{\mathcal{F}}^{\bullet}, \widetilde{\sigma})$, where $\widetilde{\mathcal{F}}^{\bullet}$ is a $\mathbb{G}_m$-equivariant tuple of torsion-free sheaves on $X \times \mathbb{A}^1_k$ and $\widetilde{\sigma}$ is a section of certain relatively affine scheme over $X \times \mathbb{A}^1_k$. Since $X$ is smooth, there exists some closed subset $Z \subset X$ of codimension at least $2$ such that the restriction of the tuple of torsion-free $\mathcal{O}_X$-sheaves $\widetilde{\mathcal{F}}^{\bullet}|_{X \times 0}$ to the open complement $X \setminus Z$ is a tuple of vector bundles. It follows from $\mathbb{G}_m$-equivariance and openness of locally free locus that the restriction $\widetilde{\mathcal{F}}^{\bullet}|_{(X \setminus Z) \times \mathbb{A}^1_k}$ is also a tuple of vector bundles, which can be naturally viewed as a $\prod_{i=1}^b \mathrm{GL}_{r_i}$-bundle on $(X \setminus Z) \times \mathbb{A}^1_k$. The restriction $\widetilde{\sigma}|_{(X \setminus Z) \times \mathbb{A}^1_k}$ amounts to a reduction of structure group of this bundle to a $G$-bundle $\widetilde{E}$ on $X \setminus Z$ \cite[Prop. 2.18]{rho-sheaves-paper}, which agrees with the constant $G$-bundle $E|_{X \setminus Z}$ on the complement of the $0$-fiber. For any $k$-point $x$ in the open $X \setminus Z$, we have a natural restriction morphism of $\mathbb{A}^1_k$-group schemes $\mathrm{res}_x: \mathrm{Aut}(\widetilde{p}) \to \underline{Aut}(\widetilde{E})|_{x \times \mathbb{A}^1_k}$, obtained by first restricting to an automorphism of the $G$-bundle $\widetilde{E}$ on $(X \setminus Z) \times \mathbb{A}^1_k$ and then further restricting to the fiber $\underline{Aut}(\widetilde{E})|_{x \times \mathbb{A}^1_k}$ over the section $x \times \mathbb{A}^1_k$. This morphims recovers the usual $\mathrm{res}_x$ for the $G$-bundle $E$ on every point of $\mathbb{A}^1_k \setminus 0$. We just need to show that the $\mathbb{A}^1_k$-fibers of the kernel $\widetilde{K}_x \subset \mathrm{Aut}(\widetilde{p})$ have dimension $0$. By using semicontinuity of fiber dimension as in the proof of Theorem \ref{thm: main theorem}, it suffices to check this at the fiber over the origin $0 \in \mathbb{A}^1_k$. 

So we consider the $\rho$-sheaf $\widetilde{p}|_{0}$ over the origin, which restricts to a $G$-bundle $\widetilde{E}_0$ on $X \setminus Z$. Since $\widetilde{p}|_{0}$ is a closed point and $\mathcal{M}$ admits a good moduli space, the algebraic group $\mathrm{Aut}(\widetilde{p}|_0)$ is reductive. We claim that the kernel $K_x \subset \mathrm{Aut}(\widetilde{p}|_0)$ of the restriction morphism $\mathrm{res}_x: \mathrm{Aut}(\widetilde{p}|_0) \to \underline{Aut}(\widetilde{E}_0)|_{x}$ is unipotent. This will imply that $K_x$ is trivial, as a normal connected unipotent subgroup of the reductive group $\widetilde{p}|_{0}$ (notice that the characteristic of $k$ is $0$, so $K_x$ is connected \cite[pg. 9, paragraph before Thm. 1.1.8]{conrad-reductive}). Hence the unipotence claim would conclude the proof of the proposition. In order to prove the claim, it suffices to show that every geometric point $g \in K_x(k)$ is unipotent, because unipotence can be checked on geometric points as $\mathrm{char}(k) = 0$ implies that $K_x$ is smooth. By definition the algebraic group $\mathrm{Aut}(\widetilde{p}|_0)$ embeds into the product $\prod_{i=1}^b \mathrm{Aut}(\widetilde{\mathcal{F}}^i|_{0})$ of the automorphisms of the sheaves in the tuple. In turn, this product acts faithfully by termwise multiplication on the tuple of $k$-vector spaces $\left(\mathrm{End}(\widetilde{\mathcal{F}}^i|_{0})\right)_{i=1}^b$. In order to show that $g \in K_x(k)$ is unipotent, it suffices to show that it acts unipotently on each element of the tuple $\left(\mathrm{End}(\widetilde{\mathcal{F}}^i|_{0})\right)_{i=1}^b$. This will be true if we can show that $(g - \mathrm{Id})^n =0$ as a tuple of endomorphisms for some $n\gg 0$. Since the sheaves are torsion-free, it suffices to show $(g - \mathrm{Id})^n =0$ over a dense open subset, such as $X \setminus Z$. In this case we can just restrict $g$ to an automorphism of the $G$-bundle $\widetilde{E}_0$ on $X \setminus Z$ and run the proof of Lemma \ref{lemma: unipotent points} replacing $X$ by $X \setminus Z$. Notice that the only properties of $X$ that we used in the proof of Lemma \ref{lemma: unipotent points} is the fact that $H^0(X, \mathcal{O}_X) = k$ and $X$ is (geometrically) reduced. In this case we also have $H^0(X \setminus Z, \mathcal{O}_{X \setminus Z}) - H^0(X, \mathcal{O}_X) = k$ by Hartogs's lemma \cite[\href{https://stacks.math.columbia.edu/tag/0E9I}{Tag 0E9I}]{stacks-project}, since $X$ is smooth and therefore satisfies Serre's $S_2$ condition. So the proof of Lemma \ref{lemma: unipotent points} applies verbatim over $X \setminus Z$.
\end{proof}

\begin{remark}
Instead of considering the stack $\mathcal{M}$ of Gieseker semistable $\rho$-sheaves, we can take the smaller open substack $\mathfrak{Y} \subset \mathcal{M}$ of slope stable $G$-bundles. This stack admits a good moduli space by \cite{hyeon-moduli-stable}, and so Theorem \ref{thm: main theorem} applies directly. In particular we can take the closed subset $Z$ to be empty if $E$ is a slope stable $G$-bundle on $X$.
\end{remark}
\end{section}

\footnotesize{\bibliography{automorphisms_decorated.bib}}
\bibliographystyle{alpha}
  \textsc{Department of Mathematics, Columbia University,
    3990 Broadway, New York, NY 10027,
USA}\par\nopagebreak
  \textit{E-mail address}, \texttt{af3358@columbia.edu}
\end{document}